\documentclass{birkjour}
\usepackage{amsmath}
\usepackage{amsthm}
\usepackage{amssymb}
\usepackage{mathtools}
\usepackage{bm}
\usepackage{newtxtext}
\usepackage{orcidlink}
\makeatletter
\renewcommand{\@evenhead}{%
  \footnotesize \hfill \leftmark\hfill\thepage 
}
\renewcommand{\@oddhead}{%
  \footnotesize \thepage\hfill \rightmark \hfill
}
\makeatother
\theoremstyle{definition}
\newtheorem{theorem}{Theorem}[section]

\newtheorem{proposition}[theorem]{Proposition}
\def\labelenumi{(\theenumi)}
\usepackage{hyperref}
\begin{document}
%
\title[Transformation from integral operator with separable kernel to matrix in eigenvalue problem]{Transformation from integral operator with separable kernel to matrix in eigenvalue problem}
%
\author[Soma Hirai]{Soma Hirai \orcidlink{0009-0005-3496-7214}}
\address{
Graduate School of Life and Environmental Sciences, Kyoto Prefectural University\\
Kyoto 606-8522\\
Japan}
\email{s825632026@kpu.ac.jp}
%
\author[Ryoto Watanabe]{Ryoto Watanabe \orcidlink{0000-0002-5758-9587}}
\address{Unaffiliated\\
Kyoto\\
Japan}
\email{watanabe.ryoto.g31@kyoto-u.jp}
%
\author[Yuki Nishida]{Yuki Nishida}
\address{Graduate School of Life and Environmental Sciences, Kyoto Prefectural University\\
Kyoto 606-8522\\
Japan}
\email{y-nishida@kpu.ac.jp}
%
\author[Masashi Iwasaki]{Masashi Iwasaki \orcidlink{0000-0002-9914-0331}}
\address{Graduate School of Life and Environmental Sciences, Kyoto Prefectural University\\
Kyoto 606-8522\\
Japan}
\email{imasa@kpu.ac.jp}
%
\subjclass{Primary 47G10 and 47A75, Secondary 65F15}
\keywords{Integral operator, Eigenvalue, Eigenfunction, Separable kernel}
\date{}
\dedicatory{}
%
\begin{abstract}
This paper investigates the eigenvalue problem of integral operators whose kernels can be expressed as a finite sum of pairwise products of single-variable functions, 
making them separable. 
By considiering the matrix form of the separable kernel in the integral operator,  
we establish the relationship between the eigenvalues and eigenfunctions of the integral operator and the eigenpairs of a matrix.
We next generalize the eigenfunction of an integral operator based on the concept of generalized eigenvectors of matrices, 
and show that solving the Fredholm integral equation of the second kind reduces to computing matrix eigenpairs and generalized eigenvectors. 
We also provide several examples to validate our results. 
\end{abstract}
\maketitle
%
\section{Introduction}
\label{sec1}
Some integral equations appear in the analysis of various scientific fields as electromagnetism and fluid dynamics. 
One of the famous equations is the Fredholm integral equation \cite{Fredholm1903}, 
which is a linear equation with an integral over a closed domain and its solution is a function. 
The Fredholm integral equation is widely studied in relation to boundary value problems for elliptic partial differential equations \cite{Atkinson1990}. 
Actually, we can solve boundary value problems in Laplace's and Poisson's equations, 
which are fundamental partial differential equations in physics and engineering that describe steady-state phenomena and potentials, 
by solving the Fredholm integral equations instead of the original problems. 
In this process, the Green's function method is used to represent unit impulse response of a differential operator \cite{Atkinson1990}. 
So is also in the Love's equation that can determine the surface charge density on two coaxial, oppositely charged conducting disks of equal radius \cite{Love1949}.
Similarly, we can relate solving boundary value problems in other differential equations to solving the Fredholm integral equations. 
The Lippmann--Schwinger equation is a special case of the Fredholm integral equation, 
which describes two body interactions of atoms, photons or any other particles \cite{LippmannSchwinger1950}.
An integrand in the Lippmann-Schwinger equation corresponds to the wave function associated with an energy eigenvalue, 
which is the scalar value representing the total energy of a particle in a quantum mechanical system.
\par
The collocation method \cite{Prenter1973} is a numerical technique for solving the Fredholm integral equation, 
and 'collocation' refers to a set of points at which the Fredholm integral equation holds.
In the collocation method, the solution to the Fredholm integral equation is approximated by a linear combination of basis functions 
such that the Fredholm integral equation holds at collocation. 
Another method is the Galerkin method \cite{Ikebe1972} that also approximates the solution 
by a linear combination of basis functions. The linear combination are determined 
such that the difference between the left and right sides of the Fredholm integral equation is orthogonal to all basis functions.
Moreover, the Nystr\"om method \cite{Nystrom1930} transforms the Fredholm integral equation into a system of linear equations 
by directly discretizing the functions in the equation. 
\par
The separable kernel is a function expressed by a summation of separable two-variable functions \cite{Atkinson1997}. 
According to Atkinson \cite{Atkinson1997}, the Fredholm integral equation with the separable kernel can be related to a system of algebraic equations.
By the way, the homogeneous version of the Fredholm integral equation regarded as the eigenvalue equation of an integral operator \cite{karlin1964existence}. 
In other words, solving the Fredholm integral equation can be transformed to computing eigenvalues and eigenfunactions of the integral operator.
In this paper, by focusing on the case where the integral operator involving the separable kernel,  
we simplify finding eigenvalues and eigenfunctions of the integral operator as computing eigenvalues and eigenvectors of a matrix. 
Since generalized eigenvectors are considered in the case where some eigenvalues are multiple, 
we also propose the associated valuable functions for the integral operator.  
Moreover, we present the application to solving the Fredholm integral equation. 
As in the matrix case, we hereinafter refer to a pair of an eigenvalue and the corresponding eigenfunction as an eigenpair of the integral operator.
\par
The remainder of this paper is organized as follows. 
In Section \ref{sec2}, we first clarify a relationship of the integral operator with a separable kernel to a matrix in terms of eigenpairs. 
In Section \ref{sec3}, next consider a variant of the eigenfunction associated with matrix generalized eigenvectors, 
and show that the Fredholm integral equation can be solved based on computing matrix eigenpairs and generalized eigenvectors. 
In Section \ref{sec4}, we present some numerical examples  to confirm our results. 
Finally, in Section \ref{sec5}, we provide concluding remarks.
%
\section{Eigenpair relationship of integral operator and matrix}
\label{sec2}
In this section, we first explain how to represent separable kernels as in the quadratic forms. 
We next relate integral operators involving a separable kernels to matrices in terms of eigenpairs. 
\par
According to Atkinson \cite{Atkinson1997}, a two-variable function $K(x,y)$ is called the separable kernel 
if it can be expressed using single-variable functions $P_i(x)$ and $Q_i(y)$ as
\[
K(x,y)\coloneqq\sum_{i=1}^RP_i(x)Q_i(y),
\]
where $P_i(x)$ and $Q_i(y)$ are not singular on the domain ${\mathcal D}\subset\mathbb{R}$.
Obviously, for $P_i(x)$ and $Q_i(x)$, there are two sequences of linearly independent functions, denoted by $\{p_i(x)\}_{i=1}^N$ and $\{q_i(y)\}_{i=1}^M$, such as 
\begin{align*}
&P_i(x)=\sum_{j=1}^Nw_{i,j}p_j(x),\\
&Q_i(y)=\sum_{j=1}^Me_{i.j}q_j(y).
\end{align*}
Thus it follows that 
\begin{equation}
K(x,y)=\sum_{i=1}^N\sum_{j=1}^Ma_{i,j}p_i(x)q_j(y),
\label{eq:SeparableKernel}
\end{equation}
where $a_{i,j}\coloneq\sum_{k=1}^R w_{k,i} e_{k,j}$. 
Equation \eqref{eq:SeparableKernel} is the quadratic form of the separable kernel $K(x,y)$ associated with a matrix $A=(a_{i,j})$. 
Therefore we can relate $N$-times differentiable functions $F(x,y)$  to the  separable kernel $K(x,y)$. 
This is because the finite Taylor expansion of the two-variable function $F(x,y)$ is given by
\[
F(x,y)=\sum_{0\leq i+j\leq N}\dfrac{\partial^{i+j}F(0,0)}{\partial x^i\partial y^j}x^iy^j+O\left(\sqrt{(x^2+y^2)^N}\right).
\]
The replacement of $x^i$, $y^j$, $M$ and $\partial^{i+j}F(0,0)/(\partial x^i\partial y^j)$ with $p_i(x)$, $q_j(y)$, $N$ and $a_{i,j}$, respectively, 
in the two-variable function $F(x,y)$ lead to the quadratic form \eqref{eq:SeparableKernel}, except for the term $O(\sqrt{(x^2+y^2)^N})$. 
In other words, under the assumption that the term $O(\sqrt{(x^2+y^2)^N})$ has only a small effect on the two-variable function $F(x,y)$, 
we can approximate the two-variable function $F(x,y)$ using appropriate $p_i(x)$, $q_j(y)$ and $a_{i,j}$ as in the quadratic form \eqref{eq:SeparableKernel}. 
Since $O(\sqrt{(x^2+y^2)^N})=0$ in the case where the two-variable function $F(x,y)$ is a two-variable polynomial, 
we can exactly rewrite the two-variable polynomial as in the quadratic form \eqref{eq:SeparableKernel}.
\par
We now introduce the integral operator $T_K$ acting on a function $f(x)$ defined by
\begin{equation}
T_Kf(x)\coloneqq\int_{\mathcal D}K(x,y)f(y)w(y)dy,
\label{eq:IntegralOperator}
\end{equation}
where $w(x)$ is an arbitrary weight of the integral 
such that \eqref{eq:IntegralOperator} converges for any function $f(x)$ and $\mathcal{D}$ is a closed interval for integration.
Thus, it follows from \eqref{eq:SeparableKernel} and \eqref{eq:IntegralOperator} that
\begin{equation}
T_Kf(x)=\sum_{i=1}^Np_i(x)\sum_{j=1}^Ma_{i,j}\langle f, q_j\rangle,
\label{eq:Integraloperator_eq2}
\end{equation}
where $\langle f, q_i\rangle\coloneqq\int_{\mathcal D}f(x)q_i(x)w(x)dx$.
Therefore we can regard the definite integral $T_Kf(x)$ as the linear combination of $p_1(x),p_2(x),\ldots,p_N(x)$. 
Preparing an $M$-by-$N$ matrix $B:=(\langle p_i,q_j\rangle)$ and considering the case where $f(x)$ can be expressed using complex $c_i$ as
\[
f(x)=\sum_{i=1}^N c_i p_i(x), 
\]
we can rewrite \eqref{eq:Integraloperator_eq2} as 
\begin{equation}
T_Kf(x)=\bm{p}(x)AB\bm{c},
\label{eq:Integraloperator_eq3}
\end{equation}
where $\bm{p}(x)\coloneqq(p_1(x),p_2(x),\ldots,p_N(x))$ and $\bm{c}:=(c_1,c_2,\ldots,c_N)^\top$.
\par
We turn to the eigenvalue equation of the integral operator $T_K$ \cite{karlin1964existence} defined by
\begin{equation}
T_K\phi(x)=\lambda\phi(x),
\label{eq:EigenValueEquation}
\end{equation}
where $\lambda$ and $\phi(x)$ are called the eigenvalue of the integral operator $T_K$ and its corresponding eigenfunction, respectively. 
Using \eqref{eq:Integraloperator_eq3}, we can rewrite \eqref{eq:EigenValueEquation} as
\[
\bm{p}(x)(AB\bm{v}-\lambda\bm{v})=0,
\]
where $\bm{v}\in\mathbb{C}^N$.
Considering $p_1(x),p_2(x),\ldots,p_N(x)$ are linearly independent, we can easily derive
\[
AB\bm{v}=\lambda\bm{v}.
\]
To sum up, we have a relationship of the eigenpairs between the integral operator $T_K$ and the matrix product $AB$. 
%
\begin{theorem}\label{thm2.1}
The eigenvalues of the integral operator $T_K$ coincide with those of the matrix product $AB$. 
Moreover, the corresponding eigenfunctions $\phi(x)$ can be expressed 
using the eigenvectors $\bm{v}$ of the matrix product $AB$ as $\phi(x)=\bm{p}(x)\bm{v}$. 
\end{theorem}
We here emphasize that, if the matrix product $AB$ can be diagonalized, 
we can obtain all the eigenpairs of the integral operator $T_K$ by computing those of the matrix product $AB$.
%
\section{Generalized eigenfunctions of integral operator}
\label{sec3}
In this section, based on generalized eigenvectors of matrices, we propose their analogues for the integral operator $T_K$. 
We then show that the separable kernel $K(x,y)$ can be expressed in the form of sum of products of two biorthogonal functions. 
Moreover, we present an application of the expression to solving the Fredholm integral equation of the second kind. 
\par
According to Bronson-Costa \cite[Appendix A]{Bronson2014}, for $i=1,2,\dots,r$, 
there exist nonzero vectors $\bm{u}_i^{(1)},\bm{u}_i^{(2)},\ldots,\bm{u}_i^{(\ell_i)}$ satisfying
\begin{equation}
\begin{cases}\begin{aligned}
&(BA-\lambda_iI_M)\bm{u}_i^{(1)}=\bm{0},\quad i=1,2,\ldots,r,\\
&(BA-\lambda_iI_M)\bm{u}_i^{(j+1)}=\bm{u}_i^{(j)},
\quad j=1,2,\ldots,\ell_i-1,
\end{aligned}\end{cases}
\label{eq:GeneralizedEigenValue}
\end{equation}
where $I_M$ is the $M$-by-$M$ identity matrix , $\ell_1,\ell_2,\ldots,\ell_r$ denote multiplicity 
of $\lambda_1,\lambda_2,\ldots,\lambda_r$, respectively, and $\ell_1+\ell_2+\cdots +\ell_r=M$. 
Observing \eqref{eq:GeneralizedEigenValue}, we see that, for $i=1,2,\ldots,r$, $\bm{u}_{i}^{(1)},\bm{u}_{i}^{(2)},\ldots,\bm{u}_{i}^{(\ell_i)}$ 
are the generalized eigenvectors of $BA$ corresponding to the eigenvalue $\lambda_i$. 
Multiplying the both hand sides of the equations in \eqref{eq:GeneralizedEigenValue} by $A$ from the left, 
and introducing new vectors 
$\bm{v}_i^{(1)}\coloneqq A\bm{u}_i^{(1)}$, 
$\bm{v}_i^{(2)}\coloneqq A\bm{u}_i^{(2)}$, 
$\dots$, 
$\bm{v}_i^{(\ell_i)}\coloneqq A\bm{u}_i^{(\ell_i)}$ for $i=1,2,\dots,r$, we derive
\begin{equation}
\begin{cases}\begin{aligned}
&(AB-\lambda_iI_N)\bm{v}_i^{(1)}=\bm{0},\quad i=1,2,\dots,r,\\
&(AB-\lambda_iI_N)\bm{v}_i^{(j+1)}=\bm{v}_i^{(j)},\quad j=1,2,\ldots,\ell_i-1,
\end{aligned}\end{cases}.
\label{eq:GeneralizedEigenValue_eq2}
\end{equation}
which implies that $v_i^{(1)},v_i^{(2)},\ldots,v_i^{(\ell_i)}$ for $i=1,2,\dots,r$ are the generalized eigenvectors of the matrix product $AB$. 
We here emphasize that the generalized eigenvectors of $AB$ and $BA$ are related to 
each other such as $v_i^{(1)}=Au_i^{(1)}$, $v_i^{(2)}=Au_i^{(2)}$, $\dots$, $v_i^{(\ell_i)}=Au_i^{(\ell_i)}$ for $i=1,2,\dots,r$.
\par
We can rewrite \eqref{eq:GeneralizedEigenValue} and \eqref{eq:GeneralizedEigenValue_eq2} in the matrix form as
\begin{equation}
\begin{cases}\begin{aligned}
&BV=U\Sigma,\\
&AU=V,
\end{aligned}\end{cases}.
\label{eq:GeneralizedEigenValue_eq3}
\end{equation}
where $U$ is an $M$-by-$M$ matrix, $V$ is an $N$-by-$M$ matrix and $\Sigma$ is $M$-by-$M$ block diagonal matrix given by
\begin{align*}
&U\coloneqq(\bm{u}_1^{(1)},\bm{u}_1^{(2)},\dots,\bm{u}_1^{(\ell_1)},\dots,\bm{u}_r^{(1)},\bm{u}_r^{(2)},\dots,\bm{u}_r^{(\ell_r)}),\\
&V\coloneqq(\bm{v}_1^{(1)},\bm{v}_1^{(2)},\dots,\bm{v}_1^{(\ell_1)},\dots,\bm{v}_r^{(1)},\bm{v}_r^{(2)},\dots,\bm{v}_r^{(\ell_r)}),\\
&\Sigma\coloneqq{\rm diag}(\Lambda_1,\Lambda_2,\dots,\Lambda_r),\\
&\Lambda_i\coloneqq\left(\begin{array}{cccc}
\lambda_i & 1 & & \\
& \lambda_i & \ddots & \\
& & \ddots & 1\\
& & & \lambda_i
\end{array}\right),\quad i=1,2,\dots,r.
\end{align*}
Since the columns of $U$ are linearly independent, $U$ is invertible.
Preparing two vector-valued functions $\Phi(x)\coloneqq V^\top\bm{p}(x)$ and $\Psi(y)\coloneqq U^{-1}\bm{q}(y)$, 
where $\bm{q}(y)=(q_1(y),q_2(y),\ldots,q_M(y))$.
We obtain a proposition concerning the orthogonality of their entries. 
%
\begin{proposition}\label{prop3.1}
Let $\phi_i^{(j)}(x)$ and $\psi_i^{(j)}(y)$ denote the $(\sum_{k=1}^{i-1}\ell_k+j)$th entries of $\Phi(x)$ and $\Psi(y)$, respectively. 
Then, it holds that
\begin{align}
&\langle\phi_{i_1}^{(j_1)},\psi_{i_2}^{(j_2)}\rangle=(\lambda_{i_1}\delta_{j_1,j_2}+\delta_{j_1+1,j_2})\delta_{i_1,i_2},\notag\\
&\quad j_1=1,2,\dots,\ell_{i_1},\quad j_2=1,2,\dots,\ell_{i_2},\quad i_1, i_2=1,2,\dots,r,\label{eq:prop3.1}
\end{align}
where $\delta_{i,j}$ is the Kronecker delta. 
\end{proposition}
\begin{proof}
Observing $(\sum_{k=1}^{i-1}\ell_k+j)$th entries of $\Phi(x)$ and $\Psi(y)$, we obtain
\begin{align}
&\phi_i^{(j)}(x)=\bm{p}^\top(x)\bm{v}_i^{(j)},\label{eq:prop3.1_proof_eq1}\\
&\psi_i^{(j)}(y)=\bm{\hat{u}}_i^{(j)}\bm{q}(y),\label{eq:prop3.1_proof_eq2}
\end{align}
where $\bm{\hat{u}}_i^{(j)}$ is $(\sum_{k=1}^{i-1}\ell_k+j)$th raw vector of $U^{-1}$. 
We here introduce a new $M$-by-$M$ matrix $B^\ast$ whose $(\sum_{k=1}^{i_2-1}\ell_k+j_2,\sum_{k=1}^{i_1-1}\ell_k+j_1)$th entry is 
$\langle\phi_{i_1}^{(j_1)},\psi_{i_2}^{(j_2)}\rangle$. 
Then, from \eqref{eq:prop3.1_proof_eq1} and \eqref{eq:prop3.1_proof_eq2}, we can derive
\begin{equation}
B^\ast=U^{-1}BV.
\label{eq:prop3.1_proof_eq3}
\end{equation}
Using the first equation of \eqref{eq:GeneralizedEigenValue_eq3}, we can rewrite the right-hand side is in \eqref{eq:prop3.1_proof_eq3} as $\Sigma$. 
The equality of the $(\sum_{k=1}^{i_1-1}\ell_k+j_1)$th row in \eqref{eq:prop3.1_proof_eq3}, thus leads to \eqref{eq:prop3.1}. 
\end{proof}
\par
For $i=1,2,\ldots,r$,  it is obvious that $\phi_i^{(1)}(x)$ is just the eigenfunction of on the both hand side $T_K$. 
Multiplying the both hand sides in the equations of \eqref{eq:GeneralizedEigenValue_eq2} by $\bm{p}(x)^\top$ from the left, 
and using \eqref{eq:prop3.1_proof_eq1}, we derive
\[
\begin{cases}
\begin{aligned}
&\bm{p}(x)^\top AB\bm{v}_{i}^{(1)}-\lambda_i\phi_i^{(1)}(x)=\bm{0},\quad i=1,2,\dots,r,\\
&\bm{p}(x)^\top AB\bm{v}_i^{(j+1)}-\lambda_i\phi_i^{(j+1)}(x)=\phi_i^{(j)}(x),\\
&\quad j=1,2,\dots,\ell_i-1,\quad i=1,2,\dots,r.
\end{aligned}
\end{cases}
\]
Note here that we derive $T_K\phi_i^{(j+1)}(x)=\bm{p}(x)^\top AB\bm{v}_i^{(j+1)}$ from replacing 
the function $f(x)$ and the vector $\bm{c}$ with $\phi_i^{(j+1)}(x)$ and $\bm{v}_i^{(j+1)}$,respectively, in \eqref{eq:Integraloperator_eq3}, we obtain
\begin{equation}
\begin{cases}
\begin{aligned}
&T_K\phi_i^{(1)}(x)=\lambda_i \phi_i^{(1)}(x),\quad i=1,2,\dots,r,\\
&T_K\phi_i^{(j+1)}(x)=\lambda_i\phi_i^{(j+1)}(x)+\phi_i^{(j)}(x),\\
&\quad j=1,2,\ldots,\ell_i-1,\quad i=1,2,\dots,r.
\end{aligned}
\end{cases}
\label{eq:GeneralizedEigenValue_eq4}
\end{equation}
Compared to the matrix case, we hereinafter refer to $\phi_i^{(1)}(x)$, $\phi_i^{(2)}(x)$,$\dots$,$\phi_i^{(\ell_i)}(x)$ 
as the generalized eigenfunctions of the integral operator $T_K$. 
Of course, $\phi_i^{(1)}(x)$ is equal to $\phi_i(x)$ appearing in Section \ref{sec2} if $ \ell_i=1$. 
Thus, to distinguish $\phi_i^{(1)}(x)$ with $\ell_i=1$ from with $\ell_i\ne 1$, we hereinafter call it the ordinary eigenfunction.
From the quadratic form  \eqref{eq:SeparableKernel}, we can easily derive $K(x,y)=\bm{p}(x)^\top A\bm{q}(y)$, 
and then can rewrite this as $K(x,y)=\Phi(x)^\top\Psi(y)$. 
We therefore arrive at a theorem for an expression using the two function $\Phi(x)$ and $\Psi(x)$ of the separable function $K(x,y)$.
\begin{theorem}\label{thm3.2}
For $i=1,2,\dots,r$, let $\phi_i^{(j)}(x)$ be the generalized eigenfunctions of the integral operator $T_K$. 
Moreover, let $\psi_i^{(j)}(y)$ be functions biorthogonal to $\phi_i$. 
Then, the separable kernel $K(x,y)$ can be expressed as
\begin{equation}
K(x,y)=\sum_{i=1}^r\sum_{j=1}^{\ell_i}\phi_i^{(j)}(x)\psi_i^{(j)}(y).
\label{eq:thm3.2}
\end{equation}
\end{theorem}
By the way, if the eigenvalues of $T_K$ are distinct to each other, namely, 
$r=M$ and $\ell_1=1$, $\ell_2=1$, $\ldots$, $\ell_M=1$ and are nonzero in \eqref{eq:thm3.2}, 
we can also rewrite \eqref{eq:thm3.2} as  
\[
K(x,y)=\sum_{i=1}^M\lambda_i\hat{\phi}_i(x)\hat{\psi}_i(y)
\]
where $\hat{\phi}_i(x)\coloneqq (1/\sqrt{\lambda_i})\phi_i^{(1)}(x)$ 
and $\hat{\psi}_i(y)\coloneqq (1/\sqrt{\lambda_i})\psi_i^{(1)}(y)$.
\par
We now focus on the case where the separable kernel is symmetric, namely, $K(x,y)=K(y,x)$, 
which imposes the constraints $a_{i,j}=a_{j,i}$ and $q_j(y)=p_i(y)$ in the quadratic form \eqref{eq:SeparableKernel}.
Applying the Gram-Schmidt process \cite{leon2013gram} to the functions $p_1(x),p_2(x),\ldots,p_N(x)$, 
we can obtain a set of orthogonal functions $\{\hat{p}_1(x),\hat{p}_2(x),\ldots,\hat{p}_N(x)\}$, 
and can rewrite the quadratic form \eqref{eq:SeparableKernel} as
\[
K(x,y)=\sum_{i=1}^{N}\sum_{j=1}^{N}\hat{a}_{i,j}\hat{p}_i(x)\hat{p}_j(y),
\]
where $\hat{a}_{i,j}=\int_{\mathcal{D}}\int_{\mathcal{D}}p_i(x)K(x,y)p_j(y)dxdy$. 
Thus, by comparing this with the quadratic form \eqref{eq:SeparableKernel}, 
we can regard that $A=(\hat{a}_{i,j})$ and $B=\langle\hat{p}_i,\hat{p}_j\rangle$ in the above discussion. 
Noting that $\hat{a}_{i,j}=\hat{a}_{j,i}$ and $\langle\hat{p}_i,\hat{p}_j\rangle=\delta_{i,j}$, we see that $A$ is symmetric and $B=I_N$.
Therefore, since the matrix product $BA$ is symmetric, it holds that $U^{-1}=U^\top$. 
This implies that we can obtain the biorthogonal expression \eqref{eq:thm3.2} of the symmetric separable kernel $K(x,y)$ more easily, 
compared to the case where $K(x,y)$ is not symmetric.
\par
The remainder of this section describes an application of the biorthogonal expression \eqref{eq:thm3.2}. 
We focus on the Fredholm integral equation of the second kind:
\begin{equation}
T_Kf(x)=zf(x)+g(x),
\label{eq:Fredholm}
\end{equation}
where the integral operator $T_K$ involves the separable kernel $K(x,y)$, $g(x)$ is an arbitrary function and $z$ is a spectral parameter.
Considering an expansion of $f(x)$ with complex $\alpha_i^{(j)}$:
\begin{equation}
f(x)=-\frac{1}{z}g(x)+\sum_{i=1}^r\sum_{j=1}^{\ell_i}\alpha_i^{(j)}\phi_i^{(j)}(x),
\label{eq:Fredholm_eq2}
\end{equation}
and using the biorthogonal expression \eqref{eq:thm3.2}, we can rewrite the Fredholm integral equation \eqref{eq:Fredholm} as
\[
\Phi(x)\left((\Sigma-zI)\bm{\alpha}-\frac{1}{z}\bm{\beta}\right)=0,
\]
where $\bm{\alpha}\coloneq(\alpha_1^{(1)},\alpha_1^{(2)},\ldots,\alpha_r^{(\ell_r)})$ and $\bm{\beta}$ is 
an $M$-dimensional column vector whose $\sum_{k=1}^{i-1}(\ell_k+j)$th entries, denoted by $\beta_i^{(j)}$, are given as
\begin{equation}
\beta_i^{(j)}\coloneq\langle g,\psi_i^{(j)}\rangle,\quad i=1,2,\ldots,r,\quad j=1,2,\ldots,\ell_i.
\label{eq:Fredholm_eq3}
\end{equation}
Since $\phi_1^{(1)}(x),\phi_1^{(2)}(x),\ldots,\phi_r^{(\ell_r)}(x)$ are linear independent, we easily derive
\begin{equation}
(\Sigma-zI)\bm{\alpha}=\frac{1}{z}\bm{\beta}.
\label{eq:Fredholm_eq4}
\end{equation}
Remarkably, if $z$ is not equal to any of $\lambda_1,\lambda_2,\dots,\lambda_r$, 
then, computing eigenpairs and generalized eigenvectors of the integral operator $T_K$ determines the entries of $\bm{\beta}$ but not $\bm{\alpha}$. 
It is obvious that $\Sigma-zI$ is nonsingular.
Thus, it follows from \eqref{eq:Fredholm_eq4} that $\bm{\alpha}=(1/z)(\Sigma-zI)^{-1}\bm{\beta}$. 
Therefore from \eqref{eq:Fredholm_eq2}, we obtain a solution to the Fredholm integral equation \eqref{eq:Fredholm}.
The well-known degenerate kernel method differs from our approach 
in that it transforms the Fredholm integral equation \eqref{eq:Fredholm} to a system of algebraic equations. 
In finite arithmetic, we emphasize that computing all eigenpairs of a matrix is generally better choice than finding all solutions to a system of algebraic equations. 
Recall here that, to obtain the biorthogonal expression \eqref{eq:thm3.2}, 
we can reduce computing the matrix inverse to employing the matrix transpose in the case where the separable kernel $K(x,y)$ is symmetric. 
Therefore, our approach is effective for accurately solving 
the Fredholm integral equation \eqref{eq:Fredholm} with the symmetric $K(x,y)$ in finite arithmetic.
%
\section{Numerical Examples}
\label{sec4}
In this section, we first present two numerical examples to show that 
eigenpairs of integral operator are related to eigenpairs of matrices shown as in Theorem \ref{thm2.1} 
and the separable kernels can be expressed using their generalized eigenfunctions shown as in Theorem \ref{thm3.2}. 
In the first example, all eigenvalues of the integral operator are distinct to each other. 
In contrast, some of the eigenvalues are multiple in the second example.
Moreover, we demonstrate that Theorem \ref{thm3.2} is the useful to solve the Fredholm integral equation \eqref{eq:Fredholm} 
involving an integral operator, with multiple in third example, with distinct eigenvalues in fourth example.
We here recall that the eigenfunction $\phi(x)$ in Theorem \ref{thm2.1} is the same as 
the ordinary eigenfunction $\phi_i^{(1)}(x)$ in Theorem \ref{thm3.2}. 
To avoid confusion, we use only the notation $\phi_i^{(1)}(x)$ as the ordinary eigenfunction $\phi_i^{(1)}(x)$ in this section.
We also adopt the same manner for the eigenvector $u_i$ and the ordinary eigenvector $u_i^{(1)}$.
\par
We first focus on the integral operator $T_K$ with the separable kernel given by, for $0\le x\le 1$ and $0\le y\le 1$,  
\begin{align*}
  K(x,y)&=232-188y+70y^2-127y^3\\
  &-2060x+1470xy-280xy^2+970xy^3\\
  &+2900x^2-1750x^2y-20x^2y^2-860x^2y^3\\
  &+2030x^3-1645x^3y+0-1925x^3y^3\\
  &-3360x^4+2240x^4y+420x^4y^2+2030x^4y^3,
\end{align*}
with the weight function $f(x)=1$ and the integral range ${\mathcal D}=[0,1]$. 
We can rewrite the separable kernel $K(x,y)$ as the quadratic form \eqref{eq:SeparableKernel} with
\arraycolsep=2pt
\[
\begin{aligned}
&\bm{p}(x)\coloneq(1,x,x^2,x^3,x^4)^\top,\\
&\bm{q}(y)\coloneq(1,y,y^2,y^3)^\top,\\
&A\coloneq\left(\begin{array}{cccc}
232 & -188 & 70 & -127\\ 
-2060 & 1470 & -280 & 970\\ 
2900 & -1750 & -20 & -860\\ 
2030 & -1645 & 0 & -1925\\ 
-3360 & 2240 & 420 & 2030
\end{array}\right).
\end{aligned}
\]
\arraycolsep=5pt
Then, we can easily derive the $(i,j)$ entry of $B$ as $\langle p_i,q_j\rangle=1/(i+j-1)$. 
Thus, it follows that
\def\arraystretch{1.2}
\begin{align*}
BA=\left(\begin{array}{cccc}
\frac{25}{6} & \frac{5}{12} & \frac{22}{3} & \frac{-47}{12}\\
\frac{1}{3} & \frac{17}{6} & \frac{20}{3} & \frac{-11}{6}\\
\frac{2}{3} & \frac{2}{3} & \frac{28}{3} & \frac{-8}{3}\\
\frac{-2}{3} & \frac{1}{3} & \frac{32}{3} & \frac{-7}{3}
\end{array}\right).
\end{align*}
\def\arraystretch{1}
The eigenvalues of $BA$ are $\lambda_1=8$, $\lambda_2=4$, $\lambda_3=2$, $\lambda_4=0$ 
and the corresponding ordinary eigenvectors are respectively
$\bm{u}_1^{(1)}=(1,1,1,1)^\top$, $\bm{u}_2^{(1)}=(1,-1,-1/2,-1)^\top$, $\bm{u}_3^{(1)}=(1/4,-1,1/4,1/2)\top$, 
$\bm{u}_4^{(1)}=(1/2,0,1/4,1)^\top$.
According to Horn and Johnson \cite[Theorem 1.3.22]{HornJohnson2012}, the eigenvalues of $AB$ coincide with those of $BA$. 
Thus, since the eigenvalues of $AB$ are distinct to each other, we can use Theorem \ref{thm2.1}, 
and see that eigenvalues of the separable kernel $T_K$ are $\lambda_1=8$, $\lambda_2=4$, $\lambda_3=2$, $\lambda_4=0$.
Since $\bm{v}_1^{(1)}=A\bm{u}_1^{(1)}=(-13,100,270,-1540,1330)^\top$, 
$\bm{v}_2^{(1)}=A\bm{u}_2^{(1)}=(512,-4360,5520,5600,-7840)^\top$, 
$\bm{v}_3^{(1)}=A\bm{u}_3^{(1)}=(200,-1570,2040,1190,-1960)^\top$, 
$\bm{v}_4^{(1)}=A\bm{u}_4^{(1)}=(13/2,-130,585,-910,455)^\top$,
we can determine the ordinary eigenfunctions $\phi_1^{(1)}(x),\phi_2^{(1)}(x),\phi_3^{(1)}(x)$ and $\phi_4^{(1)}(x)$ as
\[
\begin{aligned}
&\phi_1^{(1)}(x)=-13+100x+270x^2-1540x^3+1330x^4,\\ 
&\phi_2^{(1)}(x)=512-4360x+5520x^2+5600x^3-7840x^4,\\
&\phi_3^{(1)}(x)=200-1570x+2040x^2+1190x^3-1960x^4,\\
&\phi_4^{(1)}(x)=\frac{13}{2}-130x+585x^2-910x^3+455x^4.
\end{aligned}
\]
We can easily check that
\[
\begin{aligned}
\int_{0}^{1}K(x,y)\phi_1^{(1)}(y)dy&=8\cdot\phi_1^{(1)}(x),\\
\int_{0}^{1}K(x,y)\phi_2^{(1)}(y)dy&=4\cdot\phi_2^{(1)}(x),\\
\int_{0}^{1}K(x,y)\phi_3^{(1)}(y)dy&=2\cdot\phi_3^{(1)}(x),\\
\int_{0}^{1}K(x,y)\phi_4^{(1)}(y)dy&=0.\\
\end{aligned}
\]
This confirms that $\lambda_1,\lambda_2,\lambda_3$ and $\lambda_4$ are eigenvalues of $T_K$ 
and $\phi_1^{(1)}(x),\phi_2^{(1)}(x),\phi_3^{(1)}(x),\phi_4^{(1)}(x)$ are the corresponding ordinary eigenfunctions. 
Since the inverse of $U=(\bm{u}_1^{(1)},\bm{u}_2^{(1)},\bm{u}_3^{(1)},\bm{u}_4^{(1)})$ becomes
\def\arraystretch{1.2}
\[
U^{-1}=\left(\begin{array}{cccc}
\frac{1}{4} &\frac{1}{8} & 1 &\frac{-3}{8}\\
\frac{7}{12} & \frac{-1}{24} & \frac{-1}{3} & \frac{-5}{24}\\
\frac{-1}{3} & \frac{-5}{6} & \frac{4}{3} & \frac{-1}{6}\\
\frac{1}{2} & \frac{1}{4} &-2 & \frac{5}{4}
\end{array}\right),
\] 
\def\arraystretch{1.0}
with \eqref{eq:prop3.1_proof_eq2}, we can derive four biorthogonal functions
\[
\begin{aligned}
\psi_1^{(1)}(y)&=\frac{1}{4}+\frac{1}{8}y+y^2 -\frac{3}{8}y^3,\\
\psi_2^{(1)}(y)&=\frac{7}{12}-\frac{1}{24}y-\frac{1}{3}y^2-\frac{5}{24}y^3,\\
\psi_3^{(1)}(y)&=-\frac{1}{3}-\frac{5}{6}y+\frac{4}{3}y^2-\frac{1}{6}y^3,\\
\psi_4^{(1)}(y)&=\frac{1}{2}+\frac{1}{4}y-2y^2+\frac{5}{4}y^3.
\end{aligned}
\]
Thus we can obtain the biorthogonal expansion \eqref{eq:thm3.2} with $r=4$ and $\ell_i=1$ in Theorem \ref{thm3.2}.
\par
Adopting $w(x)=x$, instead of $w(x)=1$ as the weight function yields
\[
BA=\left(\begin{array}{cccc}
-3.33\times10^{-1} & -1.97\times10^{2} & 6.67 & -1.83\\
1.67\times10^{-1} & -1.67\times10^{2} & 9.33 & -2.67\\
-1.07 & -1.42\times10^{2} & 1.07\times10^1 & -2.33\\
-2.56 & -1.24\times 10^2 & 1.11\times10^1 & -1.66
\end{array}\right)
\]
where the entries are rounded to three significant figures.
Of course, the integral operator with $w(x)=x$ differ from with $w(x)=1$, and the matrix product $BA$ in the case where $w(x)=x$, 
does not coincide with in the case where $w(x)=1$. 
The same applies to the eigenvectors of the two integral operators.
\par
We next turn to the case where the separable kernel as
\[
\begin{aligned}
 K(x,y)=&-56-153y+182y^2-194y^3\\
&+720x+1120xy-1340xy^2+2180xy^3\\
&-990x^2-910x^2y+1620x^2y^2-3700x^2y^3\\
&-1400x^3-2380x^3y+560x^3y^2-1680x^3y^3\\
&+1890x^4+2450x^4y-980x^4y^2+3780x^4y^3,
\end{aligned}
\]
with the weight function $w(x)=1$ and the integral range ${\mathcal D}=[0,1]$. 
Similar to in the first example, we can derive the quadratic form \eqref{eq:SeparableKernel} with 
\[
\begin{aligned}
&\bm{p}(x)\coloneq(1,x,x^2,x^3,x^4)^\top,\\
&\bm{q}(y)\coloneq(1,y,y^2,y^3)^\top,\\
&A\coloneq\left(\begin{array}{cccc}
-56 & -153 & 182 & -194\\
720 & 1120 & -1340 &2180\\
-990 & -910 & 1620 & -3700\\
-1400 & -2380 &560 & -1680\\
1890 & 2450 & -980 & 3780
\end{array}\right),
\end{aligned}
\]
and
\def\arraystretch{1.2}
\[
BA=\left(\begin{array}{cccc}
2 & \frac{-4}{3} & -4 & \frac{-4}{3}\\
\frac{-1}{2} & \frac{5}{3} & -2 & \frac{-4}{3}\\
0 & \frac{1}{3} & 3 & \frac{1}{3}\\
\frac{5}{4} & \frac{1}{3} & 5 & \frac{10}{3}
\end{array}\right).
\]
\def\arraystretch{1.0}
The eigenvalues of $BA$ are $\lambda_1=\lambda_2=3$, $\lambda_3=\lambda_4=2$, 
and the corresponding generalized eigenvectors are respectively
$\bm{u}_1^{(1)}=(1,1,-1/4,-1)^\top$, $\bm{u}_2^{(1)}=(0,1,0,-1)^\top$, $\bm{u}_3^{(1)}=(1,1,-1/4,-1/4)^\top$, $\bm{u}_3^{(2)}=(0,-1/2,1/4,-1)^\top$.
The second example differs from the first example in that the eigenvalues of $BA$ are not distinct to each other. 
We note that $\bm{u}_1^{(1)}$ and $\bm{u}_2^{(1)}$ are two ordinary eigenvectors corresponding to $\lambda_1$ and $\lambda_2$, whereas
$\bm{u}_3^{(1)}$ is only one ordinary eigenvector corresponding to $\lambda_3$ and $\lambda_4$. 
Thus, $\bm{u}_3^{(2)}$ is the generalized eigenvector corresponding to $\lambda_3$ and $\lambda_4$ 
that we can compute by solving the system of algebraic equations \eqref{eq:GeneralizedEigenValue}.
Since 
$\bm{v}_1^{(1)}=A\bm{u}_1^{(1)}=(-121/2,-5,1395,-2240,805)^\top$, 
$\bm{v}_2^{(1)}=A\bm{u}_2^{(1)}=(41,-1060,2790,-700,-1330)^\top$, 
$\bm{v}_3^{(1)}=A\bm{u}_3^{(1)}=(-206,1630,-1380,-3500,3640)^\top$,
Theorem \ref{thm2.1} then leads to the ordinary eigenfunctions $\phi_1^{(1)}(x),\phi_2^{(1)}(x),\phi_3^{(1)}(x)$, 
expect for the generalized eigenfunction $\phi_3^{(2)}(x)$ related to $\bm{u}_3^{(2)}$, of the integral operator $T_K$:
\[
\begin{aligned}
\phi_1^{(1)}(x)&=-\dfrac{121}{2}-5x+1395x^2-2240x^3+805x^4,\\
\phi_2^{(1)}(x)&=41-1060x+2790x^2-700x^3-1330x^4,\\
\phi_3^{(1)}(x)&=-206+1630x-1380x^2-3500x^3+3640x^4.
\end{aligned}
\]
Since $\bm{v}_3^{(2)}=A\bm{u}_3^{(2)}=(316,-3075,4560,3010,-5250)^\top$,  
we can obtain the generalized eigenfunction $\phi_3^{(2)}(x)$ of the integral operator $T_K$ using \eqref{eq:GeneralizedEigenValue_eq4}:
\[
\phi_3^{(2)}(x)=316-3075x+4560x^2+3010x^3-5250x^4.
\]
These four functions satisfy
\[
\begin{aligned}
\int_{0}^{1}K(x,y)\phi_1^{(1)}(y)dy&=3\cdot\phi_1^{(1)}(x),\\
\int_{0}^{1}K(x,y)\phi_2^{(1)}(y)dy&=3\cdot\phi_2^{(1)}(x),\\
\int_{0}^{1}K(x,y)\phi_3^{(1)}(y)dy&=2\cdot\phi_3^{(1)}(x),\\
\int_{0}^{1}K(x,y)\phi_3^{(2)}(y)dy&=2\cdot\phi_3^{(2)}(x)+\phi_3^{(1)}(x).\\
\end{aligned}
\]
which implies that Theorem \ref{thm2.1} holds. Combining the inverse matrix
\[
U^{-1}=\left(\begin{array}{cccc}
-1& \frac{-4}{3} & -8 & \frac{-4}{3}\\
\frac{-1}{2} & 1 & 2 & 0\\
2 & \frac{4}{3} & 8 & \frac{4}{3}\\
1 & 0 & 4 & 0 \\
\end{array}\right)
\]
with \eqref{eq:prop3.1_proof_eq2}, we can derive
\[
\begin{aligned}
\psi_1^{(1)}(y)=&-1-\frac{4}{3}y-8y^2-\frac{4}{3}y^3,\\
\psi_2^{(1)}(y)=&-\frac{1}{2}+y+2y^2,\\
\psi_3^{(1)}(y)=&2+\frac{4}{3}y+8y^2+\frac{4}{3}y^3,\\
\psi_3^{(2)}(y)=&1+4y^2,
\end{aligned}
\]
which leads to the biorthogonal expansion \eqref{eq:thm3.2} with $\ell_1=1$, $\ell_2=1$, $\ell_3=2$ and $r=3$ in Theorem \ref{thm3.2}.
\par
We next apply our method to solving the Fredholm integral equation:
\begin{equation}
T_Kf(x)=5f(x)+660x^2+420x^3-5880x^4+5040x^5,
\label{eq:Example3}
\end{equation}
where the separable kernel $K(x,y)$ is given by
\begin{equation}
\begin{aligned}
K(x,y)=&-64-109y+7y^2-188y^3\\
&+870x+340xy+300xy^2+2050xy^3\\
&-1590x^2+570x^2y-990x^2y^2-3780x^2y^3\\
&-700x^3-700x^3y-140x^3y^2-560x^3y^3\\
&+1680x^4-350x^4y+1050x^4y^2+2800x^4y^3,
\end{aligned}
\label{eq:Example3_eq2}
\end{equation}
with the weight function $w(x)=1$, the integral range ${\mathcal D}=[0,1]$, 
the spectral parameter $z=5$ and the function $g(x)=660x^2+420x^3-58803x^4+5040x^5$.
Along the same line as in the second example, we can see that the eigenvalues of the integral operator $T_K$ are  
$\lambda_1=\lambda_2=\lambda_3=4$, $\lambda_4=2$, and two ordinary eigenfuctions and two generalized eigenfuctions are
\[
\begin{aligned}
&\phi_1^{(1)}(x)=120.5-1330x+2685x^2-70x^3-1645x^4,\\
&\phi_2^{(1)}(x)=-306.5+3090x-5085x^2-1610x^3+4305x^4,
\end{aligned}
\]
and
\[
\begin{aligned}
&\phi_1^{(2)}(x)=-359.5+3700x-6675x^2-910x^3+4795x^4,\\
&\phi_1^{(3)}(x)=-303+3240x-5580x^2-1680x^3+4830x^4.
\end{aligned}
\]
Moreover, since the inverse of the matrix $U$ is
\[
U^{-1}=\left(\begin{array}{cccc}
-1 & 6 & 0 & -2 \\
-1 & 4 & 0 & -1 \\
1 & -2 & 2 & 0 \\
0 & 0 & -2 & 1  \\ 
\end{array}\right),
\]
we can obtain four biorthogonal functions
\[
\begin{aligned}
&\psi_1^{(1)}(y)=-1+6y-2y^3,\\
&\psi_1^{(2)}(y)=-1+4y-y^3,\\
&\psi_1^{(3)}(y)=1-2y+2y^2,\\
&\psi_2^{(1)}(y)=-2y^2+y^3.
\end{aligned}
\] 
Combining these with \eqref{eq:Fredholm_eq3}, we can derive $\bm{\beta}=(-45,-28,-5,11)$ and
\[
\Sigma=\left(\begin{array}{cccc}
4 & 1 & 0 & 0\\
0 & 4 & 1 & 0\\
0 & 0 & 4 & 0\\
0 & 0 & 0 & 2\\
\end{array}\right),
\] 
and $\Sigma-zI$ is nonsingular, and using \eqref{eq:Fredholm_eq4}, we obtain $\bm{\alpha}=(78/5,33/5,1,11/15)$.
Thus, from \eqref{eq:Fredholm_eq2}, we can construct the solution to the Fredholm integral equation \eqref{eq:Example3} 
with the separatable kernel \eqref{eq:Example3_eq2} as
\begin{equation}
f(x)=-\frac{8567}{15}+4646x-4152x^2-\frac{23044}{3}x^3+8834x^4-1008x^5.
\label{eq:Example3_eq3}
\end{equation}
We can easily check that
\[
T_K f(x)=5f(x)+g(x),
\]
which implies that $f(x)$ in \eqref{eq:Example3_eq3} is the exact solution to the Fredholm integral equation \eqref{eq:Example3} 
with the separable kernel \eqref{eq:Example3_eq2}.
\par
We finally focus on the Fredholm integral equation
\begin{equation}
T_Kf(x)=3f(x)+\cfrac{5\sqrt{2}}{4}-\cfrac{\sqrt{105}}{7}x-\cfrac{3\sqrt{2}}{2}x-\cfrac{15\sqrt{2}}{4}x^2+\cfrac{5\sqrt{2}}{2}x^3,
\label{eq:Example4}
\end{equation}
involving the symmetric separable kernel
\begin{equation}
\begin{aligned}
K(x,y)&=\cfrac{33}{8}+\left(\cfrac{-3}{4}+\cfrac{\sqrt{210}}{70}+\cfrac{\sqrt{105}}{7}+\cfrac{3\sqrt{2}}{2}\right)x
-\cfrac{27}{8}x^2+\left(\cfrac{5}{4}-\cfrac{5\sqrt{2}}{2}\right)x^3\\
&+\left(\cfrac{-3}{4}+\cfrac{\sqrt{210}}{70}+\cfrac{\sqrt{105}}{7}+\cfrac{3\sqrt{2}}{2}\right)y+\left(\cfrac{36\sqrt{210}}{35}+\cfrac{3807}{140}\right)xy\\
&+\left(\cfrac{9}{4}-\cfrac{3\sqrt{210}}{70}\right)x^2y-\left(\cfrac{135}{4}+\cfrac{6\sqrt{210}}{7}\right)x^3y\\
&-\cfrac{27}{8}y^2+\left(\cfrac{9}{4}-\cfrac{3\sqrt{210}}{70}\right)xy^2+\cfrac{81}{8}x^2y^2-\cfrac{15}{4}x^3y^2\\
&+\left(\cfrac{5}{4}-\cfrac{5\sqrt{2}}{2}\right)y^3-\left(\cfrac{135}{4}+\cfrac{6\sqrt{210}}{7}\right)xy^3-\cfrac{15}{4}x^2y^3+\cfrac{225}{4}x^3y^3,
\end{aligned}
\label{eq:Example4_eq2}
\end{equation}
\normalsize
with the weight function $w(x)=1$, the integral range ${\mathcal D}=[-1,1]$, 
the spectral parameter $z=3$ and the function $g(x)=5\sqrt{2}/4-\sqrt{105}/7x-3\sqrt{2}/2x-15\sqrt{2}/4x^2+5\sqrt{2}/2x^3$.
Applying the Gram-Schmidt process to four functions $p_1(x)=1$, $p_2(x)=x$, $p_3(x)=x^2$ and $p_4(x)=x^3$, yields
\[
\begin{aligned}
\hat{p}_1(x)&=\cfrac{\sqrt{2}}{2},\\
\hat{p}_2(x)&=\cfrac{\sqrt{6}}{2}x,\\
\hat{p}_3(x)&=-\cfrac{\sqrt{10}}{4}+\cfrac{3\sqrt{10}}{4}x^2,\\
\hat{p}_4(x)&=-\cfrac{3\sqrt{14}}{4}x+\cfrac{5\sqrt{14}}{4}x^3.\\
\end{aligned}
\]
Thus we can express the separable kernel $K(x,y)$ using $\hat{p}_1(x),\hat{p}_2(x),\hat{p}_3(x),\hat{p}_4(x)$ as
\[
\begin{aligned}
K(x,y)&=6\hat{p}_1(x)\hat{p}_1(y)+\cfrac{2\sqrt{35}}{7}\hat{p}_2(x)\hat{p}_1(y)-\cfrac{2\sqrt{14}}{7}\hat{p}_4(x)\hat{p}_1(y)\\
&+\cfrac{2\sqrt{35}}{7}\hat{p}_1(x)\hat{p}_2(y)+\cfrac{162}{35}\hat{p}_2(x)\hat{p}_2(y)
-\cfrac{2\sqrt{14}}{35}\hat{p}_3(x)\hat{p}_2(y)-\cfrac{24\sqrt{10}}{35}\hat{p}_4(x)\hat{p}_2(y)\\
&-\cfrac{2\sqrt{14}}{35}\hat{p}_2(x)\hat{p}_3(y)+\cfrac{9}{5}\hat{p}_3(x)\hat{p}_3(y)-\cfrac{2\sqrt{35}}{35}\hat{p}_4(x)\hat{p}_3(y)\\
&-\cfrac{2\sqrt{14}}{7}\hat{p}_1(x)\hat{p}_4(y)-\cfrac{24\sqrt{10}}{35}\hat{p}_2(x)\hat{p}_4(y)
-\cfrac{2\sqrt{35}}{35}\hat{p}_3(x)\hat{p}_4(y)+\cfrac{18}{7}\hat{p}_4(x)\hat{p}_4(y).
\end{aligned}
\]
Since $B=I$, it follows that
\[
BA=\left(\begin{array}{cccc}
6 & \cfrac{2\sqrt{35}}{7} & 0 & \cfrac{-2\sqrt{14}}{7}\\
\cfrac{2\sqrt{35}}{7} & \cfrac{162}{35} & \cfrac{-2\sqrt14}{35} & \cfrac{-24\sqrt{10}}{35}\\
0 & \cfrac{-2\sqrt14}{35} & \cfrac{9}{5} & \cfrac{-2\sqrt{35}}{35}\\
\cfrac{-2\sqrt{14}}{7} & \cfrac{-24\sqrt{10}}{35} & \cfrac{-2\sqrt{35}}{35} &\cfrac{18}{7}\\
\end{array}\right)
\]
with the eigenvalues $\lambda_1=8$, $\lambda_2=4$, $\lambda_3=2$ and $\lambda_4=1$.
Along the same line as in the first example, we can obtain four ordinary eigenfunctions of integral operator $T_K$,
\[
\begin{aligned}
&\phi_1^{(1)}(y)=-4\sqrt{2}\hat{p}_1(x)-\cfrac{4\sqrt{70}}{7}\hat{p}_2(x)+\cfrac{8\sqrt{7}}{7}\hat{p}_4(x),\\
&\phi_2^{(1)}(y)=2\sqrt{2}\hat{p}_1(x)-\cfrac{2\sqrt{70}}{7}\hat{p}_2(x)+\cfrac{4\sqrt{7}}{7}\hat{p}_4(x),\\
&\phi_3^{(1)}(y)=\cfrac{2\sqrt{70}}{35}\hat{p}_2(x)-\cfrac{4\sqrt{5}}{5}\hat{p}_3(x)+\cfrac{2\sqrt{7}}{7}\hat{p}_4(x),\\
&\phi_4^{(1)}(y)=\cfrac{2\sqrt{70}}{35}\hat{p}_2(x)+\cfrac{\sqrt{5}}{5}\hat{p}_3(x)+\cfrac{2\sqrt{7}}{7}\hat{p}_4(x).
\end{aligned}
\]
Moreover, by noting that $U^{-1}=U^{\top}$, we can easily obtain four biorthogonal functions
\[
\begin{aligned}
&\psi_1^{(1)}(y)=-\cfrac{\sqrt{2}}{2}\hat{p}_1(x)-\cfrac{\sqrt{70}}{14}\hat{p}_2(x)+\cfrac{\sqrt{7}}{7}\hat{p}_4(x),\\
&\psi_2^{(1)}(y)=\cfrac{\sqrt{2}}{2}\hat{p}_1(x)-\cfrac{\sqrt{70}}{14}\hat{p}_2(x)+\cfrac{\sqrt{7}}{7}\hat{p}_4(x),\\
&\psi_3^{(1)}(y)=\cfrac{\sqrt{70}}{35}\hat{p}_2(x)-\cfrac{2\sqrt{5}}{5}\hat{p}_3(x)+\cfrac{\sqrt{7}}{7}\hat{p}_4(x),\\
&\psi_4^{(1)}(y)=\cfrac{2\sqrt{70}}{35}\hat{p}_2(x)+\cfrac{\sqrt{5}}{5}\hat{p}_3(x)+\cfrac{2\sqrt{7}}{7}\hat{p}_4(x).
\end{aligned}
\]
Combining these with \eqref{eq:Fredholm_eq3} and \eqref{eq:Fredholm_eq4} leads to $\bm{\alpha}=(1/15,1/3,-2/3,-1/2)$.
Therefore, as the solution to the Fredholm integral equation \eqref{eq:Example4} with the symmetric separable kernel \eqref{eq:Example4_eq2}, we can find
\begin{equation}\label{eq:Example4_Solution}
f(x)=\cfrac{2}{5}-\cfrac{9\sqrt{2}}{8}-\left( \cfrac{3\sqrt{2}}{20}+\cfrac{4\sqrt{105}}{35}\right)x+\cfrac{27\sqrt{2}}{8}x^2+\cfrac{\sqrt{2}}{4}x^3,
\end{equation}
which satisfies
\[
T_Kf(x)=3f(x)+g(x).
\]
%
\section{Concluding remarks}
\label{sec5}
In this paper, we showed how to find eigenpairs and generalized eigenfunctions of an integral operator with a separable kernels through computing those of a matrix. 
We first expressed the separable kernel in the matrix form. 
Then we related eigenpairs of the resulting matrix to those of the integral operator in the case where the matrix eigenvalues are not multiple. 
We next did so in the multiple case by considering analogues of generalized eigenvectors of the matrix for the integral operator.
We also provided an expression of the separable kernel using a combination of biorthogonal functions. 
Moreover, we revealed that computing matrix eigenpairs and eigenfunctions can be applied to 
solving the Fredholm integral equation of the second kind with an integral operator involving a separable kernel. 
We finally presented some examples to demonstrate our findings.
\par
We emphasize here that, in our approach, 
computing matrix eigenpairs dominates almost the entire procedure of solving the Fredholm integral equation if the matrix is symmetric. 
It is well known that eigenpairs of symmetric matrices can be computed with high accuracy in finite arithmetic. 
Thus, our approach is expected to be useful for accurately solving the Fredholm integral equations associated with symmetric matrices in finite arithmetic. 
An our future work is to expand the class of the Fredholm equations that can be accurately solved using our approach. 
Another work is to estimate the gaps arising in the approximation of inseparable kernels to separable kernels in computing eigenpairs of integral operators. 
\bibliographystyle{myplain}
\bibliography{Refs}
\end{document}